\documentclass{article}

\usepackage{amsmath}
\usepackage{amsfonts}
\usepackage{amssymb}
\usepackage{amsthm}
\usepackage{graphicx}
\usepackage{placeins}
\usepackage{appendix}
\usepackage{color}
\usepackage{soul}

\usepackage[utf8]{inputenc}
\usepackage{float}
\usepackage{bbold}

\def \RR {\mathbb{R}}
\def\EE {\mathbb{E}}

\def \Y {\hat{Y}}
\def \x {x^{\star}} 

\newtheorem{prop}{Proposition}[section]

\newtheorem{defi}[prop]{Definition}
\newtheorem*{lem*}{Lemma}
\newtheorem{rem}{Remark}
\newtheorem{cor}{Corollary}[section]

\newcommand{\argmin}{\mathop{\mathrm{argmin}}}

\title{Kriging measure-valued data with sparse observations: application to nuclear safety studies }

\begin{document}

\maketitle

\begin{center}

 \underline{Florian Gossard}$^{1}$, 
  François Bachoc$^{2,5}$, Jean Baccou$^{3}$, Thibaut Le Gouic$^{4}$, Jacques Liandrat$^{4}$, Tony Glantz$^{3}$ 
\bigskip

{\it
$^{1}$ IMT, Univ. Paul Sabatier, F-31062 Toulouse, France, 
florian.gossard@math.univ-toulouse.fr \\
$^{2}$ Institut Universitaire de France (IUF), France \\
$^{3}$ ASNR, France, jean.baccou@asnr.fr, tony.glantz@asnr.fr\\
$^{4}$ Aix Marseille Univ, CNRS, Centrale Marseille, I2M, Marseille, France,
thibaut.le-gouic@centrale-marseille.fr, jacques.liandrat@centrale-marseille.fr \\
$^{5}$ Univ. Lille, CNRS, UMR 8524 - Laboratoire Paul Painlevé, F-59000 Lille, France ,francois.bachoc@univ-lille.fr
}
\end{center}
\bigskip

\begin{abstract}
This work addresses the interpolation of probability measures within a spatial statistics framework. We develop a Kriging approach in the Wasserstein space, leveraging the quantile function representation of the one-dimensional Wasserstein distance. To mitigate the inaccuracies in semivariogram estimation that arise from sparse datasets, we combine this formulation with cross-validation techniques. In particular, we introduce a variant of the virtual cross-validation formulas tailored to quantile functions. The effectiveness of the proposed method is demonstrated on a controlled toy problem as well as on a real-world application from nuclear safety.

\end{abstract}

\newpage

\section{Introduction}
Probability measure-valued data interpolation plays an important role in many applications. It arises for example in fluid mechanics to analyse the dynamics of complex phenomena \cite{hugMultiphysicsOptimalTransportation2015b} or in image processing for histogram interpolation \cite{bonneel_wasserstein_2016}. The key ingredient is a reformulation of the problem in the optimal transport framework \cite{Santambrogio2015}. Several works such as \cite{pmlr-v130-chewi21a} have also exploited the connection between interpolation and optimal transport for smooth interpolation of probability distributions. \\ 
Kriging \cite{Cressie2015} is also a popular approach to perform smooth interpolation of spatially correlated processes. It is widely employed in various disciplines such as environment monitoring or natural resources evaluation (see \cite{goovaertsGeostatisticsNaturalResources1997} for some examples) but also in industrial applications involving complex computer code simulations \cite{marminWarpedGaussianProcesses2018a}. One of the specificity of Kriging is to allow estimating the spatial correlation via the so-called correlation function or semivariogram before its integration in the construction of a predictor. However, in its original formulation, Kriging is restricted to real-valued processes. 
To tackle the prediction of more complex objects, a branch of spatial statistics called object-oriented spatial statistics has been introduced \cite{Menafoglio2017} and 
several works have proposed to extend classical Kriging to this framework essentially for $L_2$ or Hilbert space-valued process \cite{caballeroUniversalKrigingApproach2013, menafoglioUniversalKrigingPredictor2013,delicadoStatisticsSpatialFunctional2010}.
More recently, some authors have considered an extension of Kriging to non-Hilbert spaces such as Wasserstein space (e.g. \cite{gouet2015}) which is particularly suitable for probability measure interpolation. For probability measures on $\RR$, \cite{Balzanella2020} focuses on quantile function Kriging that keeps the linearity of the predictor. 
In this paper, we follow the same track as of \cite{Balzanella2020} but extend the approach  to the case of sparse observations that strongly limit the accuracy of the semivariogram estimation. To achieve that, a quantile function Kriging problem is combined with the use of cross validation techniques \cite{bachoc2013cross}. \\   
Our paper is organized as follows: Section \ref{sec:overview} is devoted to an overview of Kriging of real-valued data. Then, we introduce in Section \ref{sec:3} an extension in the case of probability measures following \cite{Balzanella2020}. To overcome the limitation of semivariogram estimation in presence of sparse observations, Section \ref{sec:cv} is focused on cross validation. In particular, we propose a formulation of the classical virtual cross validation formulas \cite{Dubrule1983, bachoc2013cross,Ginsbourger2021} suitable to deal with probability measures. Finally, Section \ref{sec:applications} provides some numerical applications on a toy example and on real data coming from nuclear safety studies involving a computationnally costly computer code.

\section{Overview on Kriging of real-valued data}
\label{sec:overview}

 In this section, it is considered that $n$ observations $(y(x_1),...,y(x_n))$ are available. They are viewed as the evaluations of an unknown real-valued function $x \in \mathcal{D}\subseteq\RR^d \mapsto y(x) \in \RR$ with $d \geq 1$ and the goal is to estimate $y(x^\star)$ at a non-observed position $x^\star \in \mathcal{D}$. 

\subsection{Kriging problem}
    The core idea of Kriging is to assume that the data $(y(x_1),...,y(x_n))$ are realizations of a real-valued spatially correlated stochastic process \cite{Cressie2015} that is written:    
    
    $$\mathbf{Y} (x)=m(x)+\mathbf{Z}(x)$$

    where $m$ is the mean of $\mathbf{Y}$ and $\mathbf{Z}$ is a zero-mean spatially correlated process. In this paper, we focus on ordinary Kriging that considers an unknown constant mean. Moreover, it is assumed that $\mathbf{Y}$ is isotropic and stationary \footnote{ Stationarity is often limited to second-order stationarity. We refer to \cite{Cressie2015} for further details.}.\\
    The spatial correlation is classically characterized by the semivariogram:

\begin{defi}
    Given a stochastic process $\{\mathbf{Y}(x),x\in\mathcal{D}\}$, its semivariogram is defined by, 
    for $x,x' \in \mathcal{D}$ $$\gamma(x,x'):=\frac{1}{2}\mathbb{E}\left[ \left|\mathbf{Y}(x)-\mathbf{Y}(x')\right|^2\right].$$  
    If $\mathbf{Y}$ is isotropic and stationary, the semivariogram only depends on the distance $\|x-x'\|=h$ with $h \in\RR$ and $\|.\|$ a suitable (classically Euclidian) norm in $\RR^d$. With a slight abuse of notations, the semivariogram can be written: 
    
    $$\gamma(h)=\frac{1}{2}\mathbb{E}_{\|x-x'\|=h}\left[ \left|\mathbf{Y}(x)-\mathbf{Y}(x')\right|^2\right].$$
\end{defi}

There is a straightforward connection between the semivariogram and a covariance function that is classically used in statistics to characterize the spatial correlation of a stochastic process.  

\begin{defi}
    Let $\{\mathbf{Y}(x),x\in\mathcal{D}\}$ be an isotropic and stationary random process with constant mean and $K$ be its associated covariance function with $K(x,x')=cov(\mathbf{Y}(x),\mathbf{Y}(x')$ for $x,x'\in\mathcal{D}$. Then the semivariogram can be computed by $$\gamma(h) = K(0) - K(h).$$
\end{defi}

Given $n$ random variables $(\mathbf{Y}(x_1),...,\mathbf{Y}(x_n))$ coming from the stochastic process, the ordinary Kriging estimator of $\mathbf{Y}$ at a new point $\x \in \mathcal{D}$ denoted by $\hat{Y}(\x)$ is the Best Linear Unbiased Predictor (BLUP) written as:
\begin{align}
    \mathbf{\Y}(\x) = \sum_{i=1}^n \bar{\lambda}_i \mathbf{Y}(x_i)
    \label{krig_pred_R}
\end{align}
where
\begin{align}
    \bar{\lambda} = \argmin_{\lambda=(\lambda_1,...,\lambda_n)} \left\{ \EE\left[\left|\mathbf{Y}(\x)-\mathbf{\Y}_{\lambda}(\x)\right|^2\right],\sum_{i=1}^n \lambda_i = 1 \right\},
    \label{lambda_krig_R}
\end{align}
with $\mathbf{\Y}_{\lambda}(\x)=\sum_{i=1}^n \lambda_i .\mathbf{Y}(x_i)$

An estimation of $y(x^\star)$ denoted $\hat{y}(x^\star)$ is therefore given by (\ref{krig_pred_R}) when replacing $\mathbf{Y}(x_i)$ by $y(x_i)$. This estimation can also be interpreted as a barycenter:

\begin{align}
    \hat{y}(\x)=\argmin_{y \in \RR}\left\{\sum_{i=1}^{n} \bar \lambda_i |y(x_i) - y|^2\right\}.
    \label{bary_R}
\end{align}
The existence and uniqueness are verified in \ref{app: proof of existence and uniqueness }.\\
This reformulation will be used for the extension of Kriging to probability measure-valued data in Section \ref{sec:3}. \\

By introducing a Lagrange multiplier $\alpha \in \RR$, we can show that $\bar{\lambda}$ is the solution of the following system:
\begin{equation}
\begin{bmatrix} \Gamma & \mathbb{1}_n \\ \mathbb{1}_n^\top & 0 \end{bmatrix}
\begin{bmatrix} \mathbb{\lambda} \\ \alpha \end{bmatrix}
= \begin{bmatrix} \gamma^{\star} \\ 1 \end{bmatrix}, \\    
\label{krig_system_R}
\end{equation}

where $\Gamma$ is the $n\times n$ matrix with $\Gamma_{i,j} = \gamma(|x_i-x_j|)$, $\gamma^{\star}$ is the column vector of size $n$ with $(\gamma^{\star})_i =\gamma(|x_i-\x|) $ and $\mathbb{1}_n$ is the column vector of ones. See \cite{Stein1999} for further details. \\

 The key step of the Kriging predictor construction is the identification of the spatial structure which is addressed in the following section.

\subsection{Identification of spatial structure}

 This identification is based on an estimation from the observations of the process. We recall in the two next sections two classical approaches to perform this estimation.

\subsubsection{Experimental semivariogram estimation}
\label{sec_expvario}

A common approach in geostatistics is to fit a semivariogram model to the experimental (or empirical) semivariogram obtained from the observations.

\begin{defi}
\label{definition N(h)}
    Given $n$ realisation $(y(x_1),...,y(x_n))$ of an isotropic and stationary real valued stochastic process $\{\mathbf{Y}(x),x\in\mathcal{D}\}$, the experimental semivariogram is defined by
    $$ \gamma_{exp}(h) = \frac{1}{2 Card(N(h))} \sum_{(i,j)\in N(h)}|y(x_i)-y(x_j)|^2$$

    where $N(h)=\{(i,j) \in \{1,...,n\}^2 \mid \ h-\epsilon \leq || x_i-x_j || \leq h+\epsilon\}$ ($\epsilon$ is a training parameter).\\

\end{defi}
 The candidates for the semivariogram fitting are then chosen within a family of valid semivariogram models. In this paper, we consider the Matérn semivariogram model \cite{Stein1999} that is flexible to take into account a large variety of spatial structures.

\begin{defi} 
The Matérn semivariogram model is written:

$$\gamma_{\sigma,l,\nu}(h) = \sigma^2\left(1- \frac{2^{1-\nu}}{\Gamma(\nu)} \left( \sqrt{2\nu}\frac{h}{l} \right)^\nu K_\nu \left( \sqrt{2\nu}\frac{h}{l} \right)\right)$$

where \( \sigma^2 \) is the standard deviation parameter, \( \nu \) is the smoothness parameter and \( l \) is the length-scale parameter. Moreover, \( K_\nu \)  is a modified Bessel function  of second order and \( \Gamma(\nu) \) is the gamma function.

\end{defi}
\begin{figure}[H]
    \centering
    \includegraphics[width=0.75\linewidth]{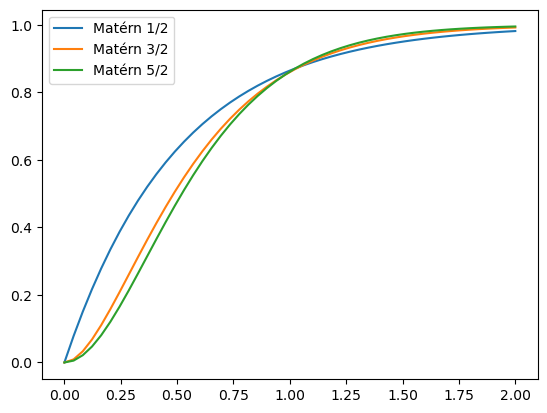}
    \caption{Matérn semivariogram models with $\sigma^2=1$,$l=0.5$ and $\nu = \frac{1}{2},\frac{3}{2},\frac{5}{2}$.}
    \label{fig:Matern_var}
\end{figure}

Figure \ref{fig:Matern_var} illustrates the Matérn semivariogram for the most classical choices for the smoothness parameter $\nu$, $\{\frac{1}{2},\frac{3}{2},\frac{5}{2}\}$. Usually, this parameter is fixed to some $\nu_0$, and the optimization is performed over $\sigma$ and $l$ by minimizing $$ \sum_{h\in \tilde{H}_{\epsilon}} \left[ \gamma_{exp}(h) - \gamma_{\sigma,l,\nu_0}(h) \right]^2$$
 where $\tilde{H}_{\epsilon} $ is a finite set of distance set that are chose as center of the bins \cite{Cressie2015}.
A major limitation of this approach lies in the challenge of accurately estimating the semivariogram for small distances when only a limited number of observations are available which is the common situation when dealing with computationally costly simulators. In this case, the covariance function of the process is estimated instead of the semivariogram by likelihood optimization  under a Gaussian process assumption. 

 \subsubsection{Maximum likelihood estimator under Gaussian process assumption}
 Here it is assumed that the random process $\mathbf{Y}$ is Gaussian with unknown constant  mean $m$ and covariance $K$. As previously, we can consider a covariance model, and a classic choice is to suppose the Matérn one. The difference compared to previously lies in the fitting method. In the Gaussian case, the smoothness parameter is still fixed to some $\nu_0$ but the other parameters are estimated by maximizing the log likelihood(\cite{Stein1999},\cite{Rasmussen2006Gaussian}) that is written for $\theta = (\sigma,l)$:
 \begin{multline*}
 l_{\nu_0}(\theta,m;\mathbf{Y}(x_1),...,\mathbf{Y}(x_n)) =  -\frac{1}{2} (\mathbf{Y}_n - m\mathbb{1}_n)^\top\mathbf{K}_{\nu_0,\theta}^{-1}(\mathbf{Y}_n - m\mathbb{1}_n)\\
 - \frac{n}{2}log(2\pi) - \frac{1}{2}log(det(\mathbf{K}_{\nu_0,\theta})),
\end{multline*}

where $\mathbf{Y}_n = (\mathbf{Y}(x_1),...,\mathbf{Y}(x_n))^\top$ and $\mathbf{K}_{\nu_0,\theta}$ is the covariance matrix of the observations, $(\mathbf{K}_{\nu_0,\theta})_{i,j} = K_{\nu_0,\theta}(x_i,x_j)$. Here $K_{\theta}$ is the covariance function corresponding to the parameter $\theta$.
One way to simplify the maximization of the log likelihood is to notice that for any fixed $\theta$, $l_{\nu_0}(\theta,m)$ is maximized by taking $\hat{m}_{\nu_0}(\theta) = \frac{\mathbb{1}^\top\mathbf{K}_{\nu_0,\theta}^{-1} \mathbf{Y}_n }{\mathbb{1}_n^\top \mathbf{K}_{\nu_0,\theta}^{-1}\mathbb{1}_n}$.
$\theta$ can now be optimized by taking the gradient of the log likelihood with $\hat{m}_{\nu_0}(\theta)$ injected, equal to zero. More details can be found in \cite{Marrel2008}\\

In the following section, the classical Kriging framework is extended to the case of probability measures prediction.

\section{Kriging of probability measure-valued processes}
\label{sec:3}

Let us first denote by $\mathcal{P}_2(\RR) $ the set of probability measures on $\RR$ with finite second moment: $\mathcal{P}_2(\RR)= \left\{ \mu \text{ proba measure} \mid \int_{\RR} x^2 \, d\mu(x) < \infty  \right\}$. \\
In this section, it is considered that $n$ measures of $\mathcal{P}_2(\RR)$, $(\mu(x_1),...,\mu(x_n))$, are available. They are viewed as the evaluations of an unknown $\mathcal{P}_2(\RR)$-valued function $x \in \mathcal{D}\subseteq\RR^d \mapsto \mu(x) \in \mathcal{P}_2(\RR)$ with $d \geq 1$ and the goal is to estimate $\mu(x^\star)$ at a non-observed position $x^\star \in \mathcal{D}$. \\

Some key notions related to probability measures on $\RR$ are first recalled. They are exploited for the construction of the Kriging extension described in Section \ref{sec:K_extension} following the work of \cite{Balzanella2020}.

\subsection{Probability measures on $\RR$}

We first recall the definition of the quantile function.
\begin{defi}
    Let $\mu$ be a probability measure of $\mathcal{P}_2(\RR) $ with cumulative distribution function defined by
$$
F_\mu(x) = \mu((-\infty, x]).
$$
The quantile function  $Q_\mu : (0,1) \to \mathbb{R}$ associated with  $\mu$ is defined as
$$
Q_\mu(\xi) = \inf \left\{ x \in \mathbb{R} : \mu((-\infty, x]) \geq \xi \right\}, \quad \text{for } \xi \in (0,1).
$$
\end{defi}
This function returns the smallest real number \( x \) such that the measure \( \mu \) of the interval \( (-\infty, x] \) is at least \( \xi \). It serves as a generalized inverse of the cumulative distribution function.\\

A relevant distance between two probability measures is the Wasserstein distance of order 2  coming from the theory of optimal transport \emph{\cite{Ambrosio2008}}. It is provided by the following definition.

\begin{defi}
\label{def_W}
    Given $\mu,\nu \in \mathcal{P}_2(\RR)$, the Wasserstein distance of order 2 is defined by:
    $$W_2(\mu, \nu) = \inf_{\pi \in \Pi(\mu, \nu)} \left( \int_{\mathcal{\RR} \times \mathcal{\RR}} \|x -  y \|^2 \, d\pi(x, y) \right)^{\frac{1}{2}},$$
where $\Pi(\mu,\nu)$ is the set of all couplings (joint distributions on $\RR\times\RR$) with marginals $\mu$ and $\nu$.\\

    This distance can be also evaluated from quantile functions:
    $$W_2(\mu, \nu) = \|Q_\mu-Q_\nu\|_{L_2([0,1])}=\left(\int_0^1 \left( Q_\mu(\xi) - Q_\nu(\xi)\right)^2 d\xi\right)^{\frac{1}{2}}.$$ 
\end{defi}

\subsection{Kriging extension to probability measure-valued data}
\label{sec:K_extension}

The key ingredient of this extension is the reformulation $\eqref{bary_R}$ of the Kriging predictor as a barycenter. Replacing $|.|$ on $\RR$ by the Wasserstein distance on $\mathcal{P}_2(\RR)$, an estimation of $\mu(x^\star)$ can be formally introduced as the following Wasserstein barycenter:

\begin{align}
    \hat{\mu}(\x) = \argmin_{\mu \in \mathcal{P}_2(\RR)}\left\{\sum_{i=1}^{n} \bar \lambda_i W_2^2(\mu(x_i),\mu)\right\},
    \label{wass_krig}
\end{align}

and a first idea to perform the extension would be to consider $\mathcal{P}_2(\RR)$-valued processes as suggested in \cite{gouet2015}. This estimation depends on the weights $\bar\lambda_1,\dots,\bar\lambda_n \in\RR$, that are addressed below. In the case of probability measures on $\RR$, a more tractable computation relies on the construction of a barycenter of quantile functions instead of probability measures since the Wasserstein distance of order $2$ induces a $L_2([0,1])$ geometry on quantile functions (Definition \ref{def_W}). \\

The advantage of working with quantile functions is that the $L_2([0,1])$-barycenter of a set of observed quantile functions is a linear combination of the observed quantile functions. Therefore, the construction of the Kriging predictor for real-valued process can be simply extended as follows:

\begin{defi}
\label{prop_ext}
Let $(\mu(x_1),...,\mu(x_n))$ be $n$ observed measures of $\mathcal{P}_2(\RR)$. Assume that the associated observed quantile functions $(Q_{\mu(x_1)},...,Q_{\mu(x_n)})$ come from a $\mathcal{Q}$-valued spatially correlated stochastic process $\{\mathbf{Q}_{\mu(x)}, x \in \mathcal{D}\}$ where $\mathcal{Q}\subset L_2([0,1])$ is the space of quantile functions. For any new location in the input space $\x$, the following Kriging predictor is introduced:

\begin{align}
    \label{Q_estim}
\hat{\mathbf{Q}}_{\mu(\x)}=\sum_{i=1}^n \bar\lambda_i \mathbf{Q}_{\mu(x_i)}
\end{align}
where
\begin{align}
    \bar{\lambda} = \argmin_{\lambda=(\lambda_1,...,\lambda_n)} \left\{ \EE\left[\int_0^1\left(\mathbf{Q}_{\mu(\x)}(\xi)-\hat{\mathbf{Q}}_{\lambda,\mu(\x)}(\xi)\right)^2 d\xi\right],\sum_{i=1}^n \lambda_i = 1 \right\},
    \label{min_system}
\end{align}
with $\hat{\mathbf{Q}}_{\lambda,\mu(\x)} = \sum_{i=1}^n \lambda_i \mathbf{Q}_{\mu(x_i)}$.

\end{defi}

The estimation of $Q_{\mu(\x)}$ is therefore obtained by replacing $\mathbf{Q}_{\mu(x_i)}$ by $Q_{\mu(x_i)}$ in (\ref{Q_estim}).

\begin{rem}
    In order to ensure that the resulting estimator is indeed a quantile function, one can impose a positivity constraint on the Kriging weights $\bar\lambda_1,\dots,\bar\lambda_n$ in \eqref{Q_estim}. This condition guarantees the monotonicity required for any quantile function. However, as shown in the numerical experiments in Section \ref{sec:applications}, this constraint can be overly restrictive for certain types of data. In practice, a suitable post-processing of the estimator’s outputs yields better results, without the need to enforce such a strict condition explicitly.

\end{rem}

\subsection{Semivariogram for probability measures}

As in the real case, a semivariogram for probability measures can be computed by exploiting the $L_2([0,1])$-distance between quantile functions.

\begin{defi}
\label{def_vario_W}
        Given a $\mathcal{Q}$-valued stochastic process $\{\mathbf{Q}_{\mu(x)},x\in\mathcal{D}\}$ associated with a $\mathcal{P}_2(\RR)$-valued function $\mu(x)$, $x\in \mathcal{D}$, then for any $x,x' \in \mathcal{D}$, we define its semivariogram by $$\gamma^{\text{\tiny W}}(x,x'):=\frac{1}{2}\mathbb{E}\left[ \int_0^1\left(\mathbf{Q}_{\mu(x)}(\xi)-\mathbf{Q}_{\mu(x')}(\xi)\right)^2d\xi\right].$$
    If $\{\mathbf{Q}_{\mu(x)}, x \in \mathcal{D}\}$ is isotropic and stationary, the semivariogram for a distance $h \in\RR$ can be written as follows:
    $$\gamma^{\text{\tiny W}}(h)=\frac{1}{2}\mathbb{E}_{\|x-x'\|=h}\left[ \int_0^1\left(\mathbf{Q}_{\mu(x)}(\xi)-\mathbf{Q}_{\mu(x')}(\xi)\right)^2d\xi\right].$$

Moreover, given $n$ observed measures $(\mu(x_1),...,\mu(x_n))$, the semivariogram can be estimated by the following experimental semivariogram:

    \begin{equation}
        \gamma^{\text{\tiny W}}_{exp}(h) = \frac{1}{2 Card(N(h))} \sum_{(i,j)\in N(h)}\left[\int_0^1\left(Q_{\mu(x_i)}(\xi)-Q_{\mu(x_j)}(\xi)\right)^2d\xi\right],
        \label{semivar_exp}
    \end{equation}
where $N(h)$ is as in Definition \ref{definition N(h)}.
\end{defi}

Once the experimental semivariogram has been computed, one may proceed with fitting a parametric model chosen from standard families commonly used in geostatistics (e.g. the Matérn class).\\

Finally, combining Definition \ref{prop_ext} and Definition \ref{def_vario_W}, the Kriging weights are solutions of the same matrix formulation as in the real case: 
\begin{equation}
    \begin{bmatrix} \Gamma_{\text{\tiny W}} & \mathbb{1}_n \\ \mathbb{1}_n^\top & 0 \end{bmatrix}
\begin{bmatrix} \mathbb{\lambda} \\ \alpha \end{bmatrix}
= \begin{bmatrix} \gamma_{\text{\tiny W}}^{\star} \\ 1 \end{bmatrix}, \\
\label{sigma_wass}
\end{equation}

where $\Gamma_{\text{\tiny W}}$ is the $n\times n$ matrix with $(\Gamma_{\text{\tiny W}})_{i,j} = \gamma^{\text{\tiny W}}(|x_i-x_j|)$, $\gamma_{\text{\tiny W}}^{\star}$ is the column vector of size $n$ with $(\gamma_{\text{\tiny W}}^{\star})_i =\gamma^{\text{\tiny W}}(|x_i-\x|) $ and $\mathbb{1}_n$ is the column vector of ones. The originality of such a formulation stands in the semivariogram model that now takes into account the complexity of the measure space. We prove Equation \ref{sigma_wass} in Appendix \ref{app: proof of the Kriging system solution} \\

    As mentioned in Section \ref{sec_expvario}, the experimental semivariogram estimation is known to suffer from significant limitations when the number of available observations is small, leading to unstable or imprecise estimates of the spatial structure. Moreover, in the case of probability measures, maximum likelihood estimation is not applicable as Gaussianity itself is not defined canonically.
    To overcome this limitation, we propose in the following section the use of cross-validation techniques and present some extension to the Kriging of probability measures.
    
\section{Cross validation techniques for Kriging of  probability measures and extension of virtual cross validation formulas}
\label{sec:cv}

Cross validation techniques are widely used in statistical prediction and they typically provide good results, even when a limited number of observations are available. In our context, these methods allow estimating the semivariogram parameters without considering the experimental semivariogram. In the Kriging framework, we are particularly interested in the Leave One Out (LOO) technique.

\subsection{LOO Cross Validation}

Let $(\mu(x_1),...,\mu(x_n))$ be $n$ observed measures of $\mathcal{P}_2(\RR)$. Assume that the associated observed quantile functions $(Q_{\mu(x_1)},...,Q_{\mu(x_n)})$ come from a $\mathcal{Q}$-valued spatially correlated stochastic process $\{\mathbf{Q}_{\mu(x)}, x \in \mathcal{D}\}$ where $\mathcal{Q}\subset L_2([0,1])$ is the space of quantile functions.
    The LOO cross-validation procedure consists in computing the prediction $ \hat{Q}^{(-i)}_{\mu(x_i)} $ of $Q_{\mu(x_i)}$, using the observations at all other locations $\{x_j\}_{j \ne i} $, for each $ i \in \{1, \dots, n\} $, and for fixed model parameters.

For each prediction, an error metric is computed between $ \hat{Q}^{(-i)}_{\mu(x_i)} $ and the true value $Q_{\mu(x_i)}$. The total cross-validation error is then obtained by aggregating these individual errors over all $ i $.

This process is repeated for a collection of candidate parameter values, and the set of parameters minimizing the cross-validation error is selected as optimal for the model.

For the prediction of probability measures, the natural metric to choose is the Wasserstein distance, that is the $L^2$ distance between quantile functions. The total LOO cross-validation error is :
\begin{align}
    MSE_{LOO} = \frac{1}{n} \sum_{i=1}^{n} \int_0^1\left(Q_{\mu(x_i)}(\xi)-\hat{Q}^{(-i)}_{\mu(x_i)}(\xi)\right)^2 d\xi .
    \label{MSE_LOO}
\end{align}

\subsection{Virtual cross validation formulas for quantile functions}
\label{sec:virtual}

One of the limitations of the LOO technique is the computational cost since we need to predict $n$ elements with $n-1$ observations. That means for the Kriging system (\ref{sigma_wass}) to invert $n$ times, a different $(n\times n)$ matrix. We show in the following the extension of virtual cross-validation formulas \cite{bachoc2013cross} for quantile functions Kriging that allows to compute the LOO error by only inverting a $n\times n$ matrix. 

\begin{prop}
    Let $(\mu(x_1),...,\mu(x_n))$ be $n$ observed measures of $\mathcal{P}_2(\RR)$. Assume that the associated observed quantile functions $(Q_{\mu(x_1)},...,Q_{\mu(x_n)})$ come from a $\mathcal{Q}$-valued spatially correlated stochastic process $\{\mathbf{Q}_{\mu(x)}, x \in \mathcal{D}\}$ where $\mathcal{Q}\subset L_2([0,1])$ is the space of quantile functions. Its semivariogram is denoted by $\gamma^W$ and we denote $(\Gamma_{\text{\tiny W}})_{i,j} = \gamma^{\text{\tiny W}}(|x_i - x_j|)$. Then we have:
    \begin{align}
     Q_{\mu(x_i)} - \hat{Q}_{\mu(x_i)}^{(-i)}= \sum_{\substack{j=1}}^{n} \frac{\tilde{\Gamma}_{ij}}{\tilde{\Gamma}_{i,i}} Q_{\mu(x_j)} \label{line:virtual formula}
    \end{align}
with $\tilde{\Gamma} = \Gamma_{\text{\tiny W}}^{-1} - \Gamma_{\text{\tiny W}}^{-1}\mathbb{1}_n(\mathbb{1}_n^\top\Gamma_{\text{\tiny W}}^{-1}\mathbb{1}_n)^{-1}\mathbb{1}_n^\top\Gamma_{\text{\tiny W}}^{-1}$.
\label{Virtual formula}
\end{prop}

Starting from \eqref{line:virtual formula}, the evaluation of the $L^2$ norm leads to \eqref{MSE_LOO}.
The proof, detailed in Appendix \ref{app:proof of virtual formulas}, is an extension of the one provided by \cite{Dubrule1983}.

\begin{rem}
    From a computational standpoint, evaluating the LOO cross-validation error using the naive approach requires inverting $n$ matrices of size $n\times n$. In contrast, the use of virtual LOO formulas requires only one inversion of a $n\times n$ matrix. The resulting computational gain is illustrated in Section \ref{subsec:verif_cv}.
\end{rem}

\subsection{Scale parameter}

The semivariogram models we are using are of the form $\gamma_{\sigma,\phi}(h)=\sigma^2\gamma_{\phi}(h)$. Since each $\hat{Q}^{(-i)}_{\mu(x_i)}$ only depends on $\phi$ not on $\sigma^2$, the scale parameter $\sigma^2$ of the semivariogram model is not optimized during the LOO process.
As described in \cite{bachoc2013cross} and \cite{Cressie2015} we can estimate the scale parameter by:
\begin{align}
    \hat{\sigma}^2 = \frac{1}{n}\sum_{i=1}^n\frac{\int_0^1\left(Q_{\mu(x_i)}(\xi)-\hat{Q}^{(-i)}_{\mu(x_i)}(\xi)\right)^2d\xi}{\hat\sigma_i^2},
\end{align}
where $\hat\sigma_i^2 =\EE_{\sigma^2=1,\phi=\hat{\phi}} \left[\int_0^1\left(Q_{\mu(x_i)}(\xi)-\hat{Q}^{(-i)}_{\mu(x_i)}(\xi)\right)^2d\xi\right]$.\\
$\EE_{\sigma^2=1,\phi=\hat{\phi}}$ means that the expectation is computed with the prediction made with $\sigma^2=1,\phi=\hat{\phi}$.

Another advantage of the virtual cross-validation formulas is the direct availability of the estimator $\hat\sigma_i^2$. 
\begin{cor}
    Under the assumptions of Proposition \ref{Virtual formula}:
    \begin{align}
        \hat\sigma_i^2 = - \frac{1}{\tilde{\Gamma}_{i,i}}.
    \end{align}
\label{corollary}
\end{cor}

See Appendix \ref{app:proof of virtual formulas} for the proof.

\section{Numerical applications}
\label{sec:applications}

This section first provides an application on a toy example of the probability measure-valued Kriging predictor with a focus on the numerical advantages of the use of virtual cross validation formulas. Then, in Section \ref{sec:real_data}, we evaluate the performances of our approach on a real test case.

\subsection{Toy example}
\label{subsec:verif_cv}

We consider the function $\mu$ mapping $u,v \in [0,1]\times(0,2]$, to the Gaussian measure in $\mathcal{P}_2(\RR)$ with density $$f:(u,v,.)\mapsto\frac{1}{\sqrt{2\pi v}} exp^{-\frac{1}{2}\frac{(.-u)^2}{v}}.$$

It is assumed that $n$ observations are available. They are obtained from a uniform random sampling in $ [0,1]\times(0,2]$.\\

First, we aim to demonstrate the efficiency of our Kriging model when using the LOO virtual formulas (Section \ref{sec:virtual}). Taking $n=50$, the parameters of a Matérn semi-variogram are estimated on 80\% of the available data (training set) and the remaining 20\% (test set) are subsequently used to assess the predictive performance of the Kriging estimator. For this example we choose a grid of size 100 to estimate the length scale parameter. We estimate the smoothness parameter over the set $\{\frac{1}{2},\frac{3}{2},\frac{5}{2}\}$

Figure \ref{fig:toy_ex} displays several plots of predicted quantile functions on the test set. These prediction are very accurate.

\begin{figure}[H]
    \centering
    \includegraphics[width=1\linewidth]{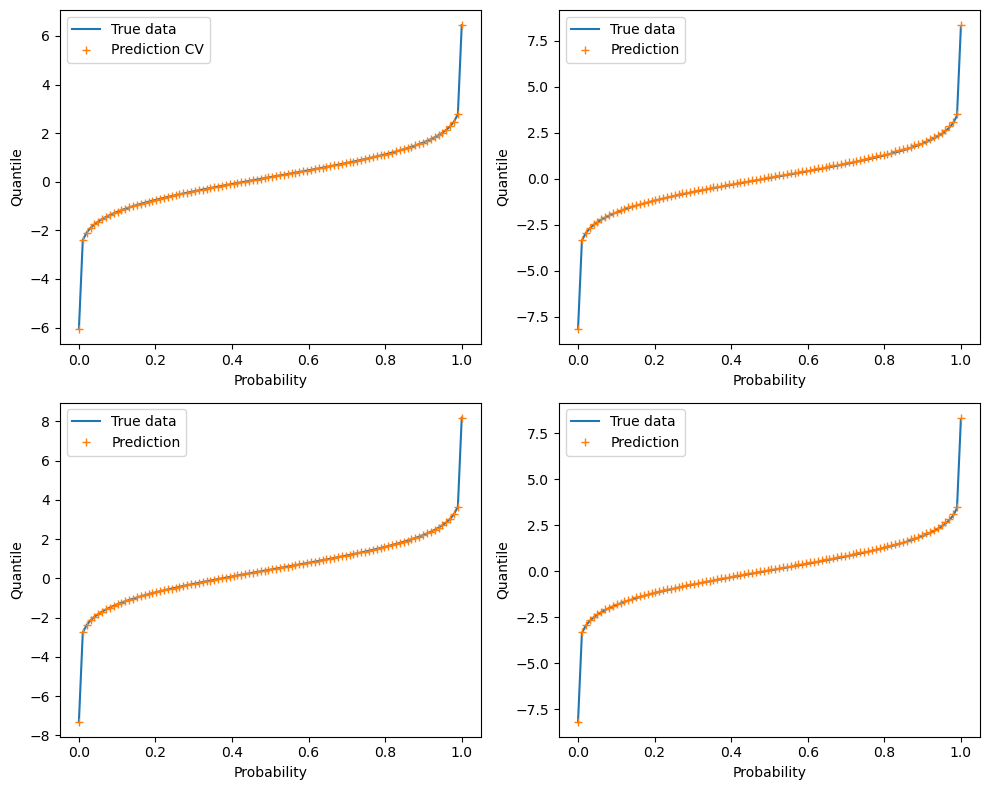}
    \caption{Examples of predicted quantile functions and comparison with true quantile functions.}
    \label{fig:toy_ex}
\end{figure}

The remaining of this section is devoted to a numerical comparison between the classical LOO and the virtual formulas. 
Figure \ref{fig:cv_formulas} shows the cross-validation errors computed using both methods. Each cross represents the cross-validation error for a given couple \textit{(length scale parameter, smoothness parameter)}: the x-coordinate corresponds to the classical method, and the y-coordinate to the virtual-formula method. The points lie along the line $y=x$, indicating that both methods yield identical results.

\begin{figure}[H]
    \centering
    \includegraphics[width=0.75\linewidth]{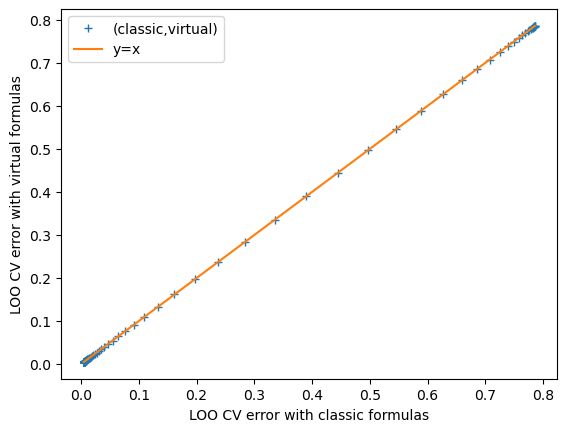}
    \caption{Verification of virtual cross validation formulas.}
    \label{fig:cv_formulas}
\end{figure}

Moreover, Table \ref{tab:comp_time} is focused on the computational cost and reports the time required to compute all cross-validation errors for both methods, as a function of the parameter $n$.

\begin{table}[!h] 
    \centering
    \begin{tabular}{|c|c|c|c|c|c|c|}
    \hline
         n              & 50 & 100 & 200 & 400&600&700\\\hline
         Classical method (s)      & 0.94 & 3.04 & 9.68 & 47.26&195.24&336.14\\\hline
         Virtual formulas method (s)& 0.21 & 0.51 & 1.39 & 5.67&11.58&16.63\\\hline
    \end{tabular}
    \caption{Computation time comparison.}
    \label{tab:comp_time}
\end{table}

This table shows that the virtual cross-validation formula significantly reduces the computational cost with a factor ranging from $4.5$ to $20$ on this example.

\subsection{Nuclear safety test case}
\label{sec:real_data}

\subsubsection{Data description}
This study investigates the numerical simulation of a Loss of Coolant Accident (LOCA) in a Pressurized Water Reactor (PWR), with particular emphasis on the modelling of the core reflooding phenomenon using the DRACCAR simulation code \cite{Bascou2015} developed by the french Autorité de Sûreté Nucléaire et de Radioprotection (ASNR). The test case consists of a time-resolved simulation of the temperature distribution across selected fuel rods in the reactor core, discretized over 100 spatial points (Figure~\ref{carte_temp}). The system configuration is defined by six physical input parameters: temperature of the injected water (\textit{dtsub} in Kelvin), external pressure (\textit{pout} in Pascal), water injection speed (\textit{vin} in meter/second), power injected into each fuel rod (\textit{p1rod} in Watt), power injected into the casing (\textit{pshroud} in Watt) and temperature of the fuel rods at the moment of water injection (\textit{tclad0} in Kelvin). Their variation domains are given in Table \ref{param_phys}. A total of 79 simulations are conducted by varying these parameters according to a random sampling from a uniform law for each input parameter, enabling a systematic exploration of the model's input space.

\begin{table}[!h] 
 \begin{center} 
\begin{tabular}{|c|c|c|c|}
\hline              &   min         &   max         &   mean        \\ \hline
dtsub (K)               &   0.217       &   59.578      &   27.655      \\ \hline
pout ($\times 10^6$) (Pa) &   0.222       &   5.968       &   3.135       \\ \hline
vin (m/s)               &   0.016       &   0.079       &   0.042       \\ \hline
p1rod (W)              &   612.542     &   3440.880    &   2077.480    \\ \hline
pshroud (W)            &   344.640     &   24407.170   &   12738.210   \\ \hline
tclad0 (K)             &   730.183     &   971.501     &   847.148     \\ \hline
\end{tabular}
\caption{Variation domain of the physical input parameters.}
\label{param_phys}
\end{center}
\end{table}

\begin{figure}[H]
    \centering
    \includegraphics[width=0.6\linewidth]{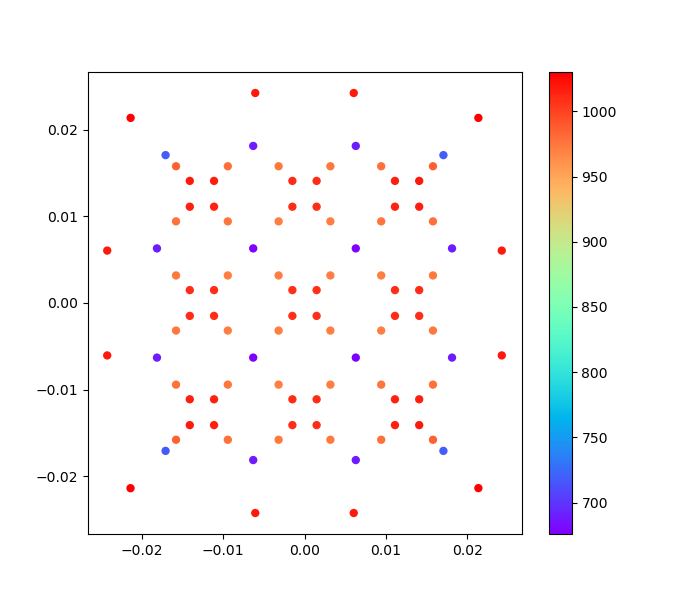}
    \caption{An example of a temperature map generated by DRACCAR (in K).}
    \label{carte_temp}
\end{figure}

 In this study, for each system configuration, we focus on a unique temperature distribution (or map) corresponding to the time when the maximal temperature is reached. In practice, it is difficult to exhaustively explore the physical input space since the computational cost of a DRACCAR simulation can vary from a few minutes to a month depending on the number of considered rods and the type of transient of interest. Therefore, we investigate in the remaining of this section how Kriging can offer an efficient strategy to construct a fast-to-evaluate model that accurately approximates the relationship between DRACCAR outputs of interest and the six physical input parameters. The outputs of interest are here statistical characteristics of the temperature map such as mean or quantiles. Instead of constructing a Kriging predictor of each statistical quantity of interest, we propose to predict the whole discrete probability measure constructed from the temperature values of the map. In the next sections, we evaluate the performances of this approach compared to classical ones.    

\subsubsection{Description of the models}
We study and compare 4 Kriging models. The first one is focused on the prediction of the temperature map exploiting $\RR$-valued data Kriging. The statistical quantities of interest are then computed from the predicted map. The three last models directly operate on the probability measures associated with the temperature maps. Full details are given below:\\

\textbf{Model based on $\RR$-valued data Kriging} (denoted SK\_PCA\_ML)\\
The idea here is to consider the temperature over each discretized point along the inputs as a $\RR$-valued Gaussian process. For this model, the predicted output is a map. An output dimension reduction is first performed by applying a Principal Component Analysis of all the points of the map. Then, a Gaussian process predictor is constructed for each selected principal component assuming a Matérn $\frac{3}{2}$ covariance function and using the maximum likelihood method for its parameter estimation \cite{marrel_2011}.  \\

\textbf{Models based on $\mathcal{P}_2(\RR)$-valued data Kriging}\\
These models are constructed following the general approach introduced in Section \ref{sec:3} and lead to a quantile function predictor. Their specificities are the following: 

\begin{itemize}
    \item Model P\_WK\_LS: its construction is taken from \cite{Balzanella2020} and relies on the estimation of a semivariogram. Moreover, a non-negativity constraint is introduced on the Kriging weights to ensure the output to be a quantile function \cite{Barnes1984}. We use a Matérn $\frac{3}{2}$ model and the least squared method for the semivariogram fitting.
    \item Model WK\_LS: it is the same as the previous model without the non-negativity constraint. We ensure the output to be a quantile function by post processing (sorting\footnote{In our numerical tests, the order relation deviations are small so that sorting does not strongly change the original predicted output.} in our case) the output after the prediction as it is performed in geostatistics when using indicator Kriging \cite{goovaertsGeostatisticsNaturalResources1997}.
    \item Model WK\_CV: its construction relies on the use of the LOO technique for semivariogram parameter estimation (Section \ref{sec:cv}). The type of semivariogram model is Matérn $\frac{3}{2}$ and we also sort the outputs to obtain quantile functions.
\end{itemize}

\subsection{Numerical results and comments}

The comparison is based on $N=100$ independent random splits of the dataset into a training set (63 training samples for each split), on which the model parameters are fitted, and a test set (16 test samples for each split), where the trained model is used to predict new data.

The $4$ Kriging models described in the previous section are considered. For SK\_PCA\_ML, we select the first 3 principal components since the average explained variance ratio on the three first components for all the different train-test sets is $0.996$. For experimental semivariogram-based models (P\_WK\_LS and WK\_LS), the distance subdivision is optimized by cross validation. Finally, the application of the LOO technique in the construction of WK\_CV is based on a grid search of size 100 for the length-scale parameter and of size 3 for the smoothness parameter.

For each split, we compute on the test set the root mean squared error (RMSE) associated with both the mean temperature (denoted $RMSE_{mean}$) and the 95\% quantile of the predicted distributions (denoted $RMSE_{q95}$). Additionally, we evaluated the predictive accuracy using the root mean squared Wasserstein distance (denoted $RMSE_W$) between the predicted and true probability measures. \\

Figures \ref{boxplot_M} to \ref{boxplot_W} display the boxplots of $RMSE_{mean}, RMSE_{q95}$ and $RMSE_{W}$ for each model.

\begin{figure}[H]
    \centering
    \includegraphics[width=0.75\linewidth]{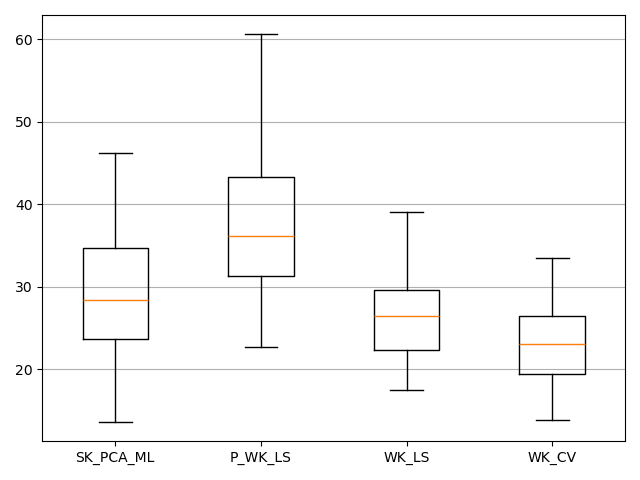} 
    \caption{Boxplots of $RMSE_{mean}$ for each model}
    \label{boxplot_M}
\end{figure}

\begin{figure}[H]
    \centering
    \includegraphics[width=0.75\linewidth]{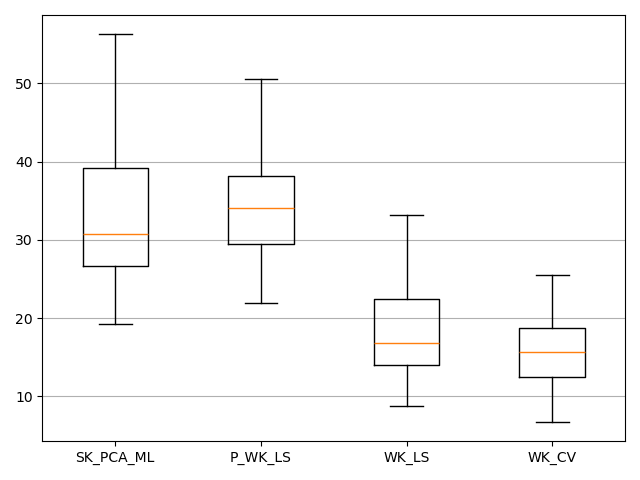}
    \caption{Boxplots of $RMSE_{q95}$ for each model}
    \label{boxplot_Q}
\end{figure}

\begin{figure}[H]
    \centering
    \includegraphics[width=0.75\linewidth]{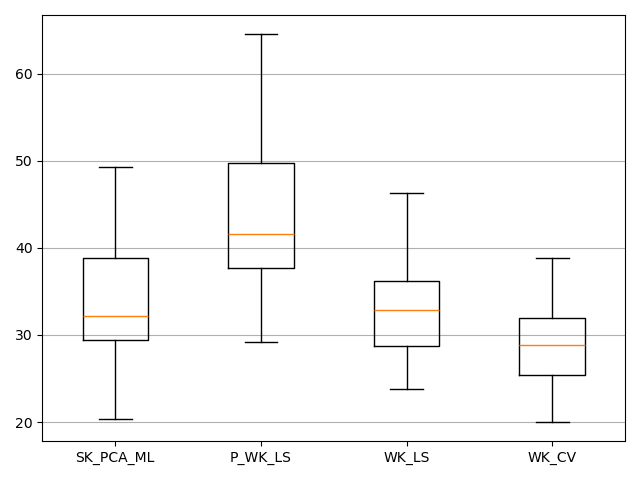}
    \caption{Boxplots of $RMSE_{W}$ for each model}
    \label{boxplot_W}
\end{figure}

We observe on these figures that the errors produced by the second method (PKW\_LS) are significantly larger than those of the other approaches. These results suggest that, for this dataset, the positivity constraint imposed on the weights is too restrictive, especially given that the same model, when applied without this constraint, yields substantially better performances.

Regarding the remaining models, we observe that the two methods based on probability Kriging (KW\_LS and KW\_CV) outperform the classical Kriging approach in $\RR$ especially on the prediction of the 95\% quantile. Furthermore, the final and most notable observation is that the probability-measure Kriging model with parameter estimation via cross-validation (KW\_CV) consistently achieves the best performances across all quantities of interest. 
\section{Conclusion}

A Kriging predictor of probability measure-valued data with sparse observations has been introduced. The originality of its construction is that it combines two main ingredients coming from the geostatistics and computer experiment frameworks. The first ingredient is an extension of classical $\RR$-valued Kriging to the case of probability measures thanks to a reformulation of the Kriging predictor as a barycenter. It therefore allows integrating suitable distance between probability measures such as the Wasserstein distance. In the case of measures on $\RR$ the Kriging problem can be formulated as a quantile function prediction. The second ingredient is the use and the adaptation of cross-validation formulas. It  avoids the estimation of an experimental semivariogram that is not affordable in presence of sparse observations. The numerical applications on the nuclear test case show that the proposed approach outperforms classical ones especially to predict extreme quantiles which are key quantities for nuclear safety. Our current further work deals with the extension of our approach to the case of anisotropic stochastic processes. The development of a Kriging predictor for probability measures on $\RR^p$ with $p>1$ will also be addressed.   

\newpage
\appendix
\section{Existence and Uniqueness of \eqref{bary_R}}
\label{app: proof of existence and uniqueness }
\begin{lem*}
    If $\sum_i^n \bar\lambda_i = 1$, then the function $y\mapsto \sum_i^n \bar\lambda_i|y(x_i) -y|^2$ admits a unique minimum
\end{lem*}
\begin{proof}
    It is straightforward to verify that the function is a quadratic polynomial, and that its leading term, $
\sum_{i=1}^n \bar{\lambda}_i y^2 = y^2,$
is positive. That gives automatically the existence and the uniqueness of a minimizer. To find this minimizer, let us derive the function and find the y for which the derivative is equal to 0. The derivative is given by $y\mapsto -2\sum_i^n\bar{\lambda}_i(y(x_i)-y)$.
So under the hypothesis of the lemma
\begin{align*}
-2\sum_i^n\bar{\lambda}_i(y(x_i)-y) = 0 \Leftrightarrow y=\sum_i^n\bar\lambda_i y(x_i) .
\end{align*}
\end{proof}

\section{Details of the resolution of the Kriging problem in $\mathcal{P}_2(\RR) $\eqref{sigma_wass}}
\label{app: proof of the Kriging system solution}
\begin{proof}[Proof of \eqref{sigma_wass}]

The ordinary Kriging estimator minimizes the estimation variance under the unbiasedness constraint ($\sum_{i=1}^n \lambda_i = 1$).
In the case of measure-valued Kriging, we want to solve equation \eqref{Q_estim}. That leads to the following constrained minimization problem:
\begin{align*}
    \min_{\lambda} \EE\left[\int_0^1\left(\mathbf{Q}_{\mu(\x)}(\xi)-\hat{\mathbf{Q}}_{\lambda,\mu(\x)}(\xi)\right)^2 d\xi\right]\\
    s.t.\  \mathbb{1}_n^\top \lambda = 1 \nonumber
\end{align*}

In the following we will denote $\mathbf{Q}_{\mu(\x)}$ by $\mathbf{Q}^{\star}$,$\hat{\mathbf{Q}}_{\lambda, \mu(\x)}$ by $\hat{\mathbf{Q}}^{\star}$ and $\mathbf{Q}_{\mu(x_i)}$ by $\mathbf{Q}_i$. 
Let us develop the expectation assuming that $\mathbb{1}_n^\top \lambda = 1 $:

\begin{align*}
    \EE\left[\int_0^1\left(\mathbf{Q}^{\star}(\xi)-\hat{\mathbf{Q}}^{\star}(\xi)\right)^2 d\xi\right] &= \EE\left[\int_0^1\left(\mathbf{Q}^{\star}(\xi)\right)^2+\left(\hat{\mathbf{Q}}^{\star}(\xi)\right)^2 - 2\mathbf{Q}^{\star}(\xi)\hat{\mathbf{Q}}^{\star}(\xi) d\xi\right]\\
    &= \EE \int_0^1\left(\mathbf{Q}^{\star}(\xi)\right)^2d\xi + \EE \int_0^1 \left(\sum_{i=1}^n \lambda_i \mathbf{Q}_i(\xi)\right)^2d\xi \\
    &\hspace{3cm}-2\EE\int_0^1\mathbf{Q}^{\star}(\xi)\sum_{i=1}^n \lambda_i \mathbf{Q}_i(\xi)d\xi\\
    &=\EE \int_0^1\left(\mathbf{Q}^{\star}(\xi)\right)^2d\xi + \sum_{i=1}^n \sum_{j=1}^n\lambda_i\lambda_j \EE \int_0^1\mathbf{Q}_i(\xi)\mathbf{Q}_j(\xi)d\xi \\
    &\hspace{3cm}-2\sum_{i=1}^n\lambda_i\EE\int_0^1\mathbf{Q}^{\star}(\xi)\mathbf{Q}_i(\xi)d\xi.
\end{align*}

In the following, we denote $T_1 =\EE \int_0^1\left(\mathbf{Q}^{\star}(\xi)\right)^2d\xi$, $T_2 =\sum_{i=1}^n \sum_{j=1}^n\lambda_i\lambda_j \EE \int_0^1\mathbf{Q}_i(\xi)\mathbf{Q}_j(\xi)d\xi $ and $T_3=-2\sum_{i=1}^n\lambda_i\EE\int_0^1\mathbf{Q}^{\star}(\xi)\mathbf{Q}_i(\xi)d\xi$.

    Let us now develop the semi-variogram formula given in Definition \ref{def_vario_W}.

\begin{align*}
    \gamma^{\text{\tiny W}}(x,x')&:=\frac{1}{2}\mathbb{E}\left[ \int_0^1\left(\mathbf{Q}_{\mu(x)}(\xi)-\mathbf{Q}_{\mu(x')}(\xi)\right)^2d\xi\right]\\
    & = \frac{1}{2} \EE\int_0^1 \left(\mathbf{Q}_{\mu(x)}(\xi)\right)^2d\xi + \frac{1}{2} \EE\int_0^1 \left(\mathbf{Q}_{\mu(x')}(\xi)\right)^2d\xi - \EE \int_0^1 \mathbf{Q}_{\mu(x)}(\xi)\mathbf{Q}_{\mu(x')}(\xi)d\xi.
\end{align*}

This gives:
\begin{align*}
    \EE \int_0^1 \mathbf{Q}_{\mu(x)}(\xi)\mathbf{Q}_{\mu(x')}(\xi)d\xi =\frac{1}{2}\left(\EE\int_0^1 \left(\mathbf{Q}_{\mu(x)}(\xi)\right)^2d\xi + \EE\int_0^1 \left(\mathbf{Q}_{\mu(x')}(\xi)\right)^2d\xi\right) - \gamma^{\text{\tiny W}}(x,x').
\end{align*}
By substituting into $T_2$ and $T_3$, we obtain:

\begin{align*}
    T_2 &=\sum_{i=1}^n \sum_{j=1}^n\lambda_i\lambda_j \left(\frac{1}{2}\EE\int_0^1 \left(\mathbf{Q}_i(\xi)\right)^2d\xi + \frac{1}{2}\EE\int_0^1 \left(\mathbf{Q}_j(\xi)\right)^2d\xi - \gamma^{\text{\tiny W}}(x_i,x_j)\right)\\
    &= - \sum_{i=1}^n \sum_{j=1}^n\lambda_i\lambda_j\gamma^{\text{\tiny W}}(x_i,x_j) + \frac{1}{2}\sum_{i=1}^n\lambda_i\EE\int_0^1 \left(\mathbf{Q}_i(\xi)\right)^2d\xi\sum_{j=1}^n\lambda_j+ \frac{1}{2}\sum_{j=1}^n\lambda_j\EE\int_0^1 \left(\mathbf{Q}_j(\xi)\right)^2d\xi\sum_{i=1}^n\lambda_i.
\end{align*}

By adding the constraint $\mathbb{1}_n^\top \lambda = 1$ we obtain:

\begin{align*}
    T_2 &= - \sum_{i=1}^n \sum_{j=1}^n\lambda_i\lambda_j\gamma^{\text{\tiny W}}(x_i,x_j) + \frac{1}{2}\sum_{i=1}^n\lambda_i\EE\int_0^1 \left(\mathbf{Q}_i(\xi)\right)^2d\xi+ \frac{1}{2}\sum_{j=1}^n\lambda_j\EE\int_0^1 \left(\mathbf{Q}_j(\xi)\right)^2d\xi\\
    &= -\sum_{i=1}^n \sum_{j=1}^n\lambda_i\lambda_j\gamma^{\text{\tiny W}}(x_i,x_j) + \sum_{i=1}^n\lambda_i\EE\int_0^1 \left(\mathbf{Q}_i(\xi)\right)^2d\xi.
\end{align*}

Also,

\begin{align*}
    T_3 &=-\sum_{i=1}^n\lambda_i \EE\int_0^1(\mathbf{Q}^{\star}(\xi))^2d\xi -\sum_{i=1}^n\lambda_i \EE\int_0^1(\mathbf{Q}_i(\xi))^2d\xi +2\sum_{i=1}^n\lambda_i\gamma^{\text{\tiny W}}(\x,x_i).\\
\end{align*}
By using the constraint and by linearity:

\begin{align*}
    T_3 &= - T_1 -\sum_{i=1}^n\lambda_i \EE\int_0^1(\mathbf{Q}_i(\xi))^2d\xi \\
    &\hspace{3cm}+2\sum_{i=1}^n\lambda_i\gamma^{\text{\tiny W}}(\x,x_i).\\
\end{align*}

Finally:
\begin{align*}
    \EE\left[\int_0^1\left(\mathbf{Q}^{\star}(\xi)-\hat{\mathbf{Q}}^{\star}(\xi)\right)^2 d\xi\right] &= T_1+T_2+T_3\\
    &= T_1 -\sum_{i=1}^n \sum_{j=1}^n\lambda_i\lambda_j\gamma^{\text{\tiny W}}(x_i,x_j) + \sum_{i=1}^n\lambda_i\EE\int_0^1 \left(\mathbf{Q}_i(\xi)\right)^2d\xi\\
    &\hspace{3cm}  -T_1 -\sum_{i=1}^n\lambda_i \EE\int_0^1(\mathbf{Q}_i(\xi))^2d\xi +2\sum_{i=1}^n\lambda_i\gamma^{\text{\tiny W}}(\x,x_i)\\
    &=-\sum_{i=1}^n \sum_{j=1}^n\lambda_i\lambda_j\gamma^{\text{\tiny W}}(x_i,x_j) +2\sum_{i=1}^n\lambda_i\gamma^{\text{\tiny W}}(\x,x_i)\\
    &= -\lambda^\top \Gamma_{\text{\tiny W}} \lambda +2\lambda^\top\gamma_{\text{\tiny W}}^{\star},
\end{align*}

where $\Gamma_{\text{\tiny W}}$ is the $n\times n$ matrix with $(\Gamma_{\text{\tiny W}})_{i,j} = \gamma^{\text{\tiny W}}(\|x_i-x_j\|)$ and $\gamma_{\text{\tiny W}}^{\star}$ is the column vector of size $n$ with $(\gamma_{\text{\tiny W}}^{\star})_i =\gamma^{\text{\tiny W}}(|x_i-\x|) $

To solve the constrained minimization problem, we introduce the Lagrange multiplier $\alpha\in\RR$ and define the Lagrangian:
$$
\mathcal{L}(\lambda,\alpha)  = -\lambda^\top \Gamma_{\text{\tiny W}} \lambda +2\lambda^\top\gamma_{\text{\tiny W}}^{\star}+ \alpha (\mathbb{1}_n^\top \lambda - 1).
$$

The solution is obtained by identifying the unique critical point of \( \mathcal{L}(\lambda, \alpha) \). 
For all \( \alpha \in \mathbb{R} \), the function \( \mathcal{L}(\cdot, \alpha) \) is convex since \( \Gamma_{\text{\tiny W}} \) is positive semi-definite. We have:

\begin{align}
    \left \{
    \begin{array}{c @{=} c}
        \frac{\partial \mathcal{L}(\lambda,\alpha)}{\partial \lambda}  &-2\Gamma_{\text{\tiny W}}\lambda + 2\gamma_{\text{\tiny W}}^{\star} +\alpha \mathbb{1}_n \\
        \frac{\partial \mathcal{L}(\lambda,\alpha)}{\partial \alpha} & \mathbb{1}_n^\top \lambda - 1.
    \end{array}
\right.
\end{align}
Then,
\begin{align*}
    \nabla \mathcal{L}(\lambda,\alpha) = 0 &\Leftrightarrow
\left \{
\begin{array}{c @{=} c}
    -2\Gamma_{\text{\tiny W}}\lambda + 2\gamma_{\text{\tiny W}} +\alpha \mathbb{1}_n & 0 \\
    \mathbb{1}_n^\top \lambda - 1 & 0.
\end{array}
\right.
\end{align*}

By setting $\beta = -2\alpha$ we can show that $\bar\lambda$ is the solution of
\begin{equation}
    \begin{bmatrix} \Gamma_{\text{\tiny W}} & \mathbb{1}_n \\ \mathbb{1}_n^\top & 0 \end{bmatrix}
\begin{bmatrix} \mathbb{\lambda} \\ \beta \end{bmatrix}
= \begin{bmatrix} \gamma_{\text{\tiny W}}^{\star} \\ 1 \end{bmatrix}.\\
\end{equation}
    
\end{proof}

\section{Proof of the virtual cross validation formulas for quantile functions}
\label{app:proof of virtual formulas}

\begin{proof}[Proof of Proposition \ref{Virtual formula}]
We begin by introducing notation to clarify the upcoming proof. 
We write $\hat Q^{\star}=\hat Q_{\mu(\x)}$, $Q_i = Q_{\mu(x_i)}$, $\mathcal{Q}^n = (Q_1,\cdots,Q_n)^\top$.

As in the scalar case \cite{Stein1999}, the vector $\bar{\lambda} \in \RR^n$  obtained by solving the ordinary Kriging problem (\ref{sigma_wass}) is:

\begin{align*}
    \bar{\lambda} = \Gamma_{\text{\tiny W}}^{-1}\gamma_{\text{\tiny W}}^{\star} +(1- \mathbb{1}_n^\top \Gamma_{\text{\tiny W}}^{-1}\gamma_{\text{\tiny W}}^{\star}) \frac{\Gamma_{\text{\tiny W}}^{-1}\mathbb{1}_n }{\mathbb{1}_n^\top \Gamma_{\text{\tiny W}}^{-1}\mathbb{1}_n}.
\end{align*}

That gives the following estimator for a fixed arbitrary $\xi \in (0,1)$:

\begin{align*}
\hat{Q}(\xi) &= \sum_{i=1}^n \lambda_i Q_i(\xi)\\
&=\bar{\lambda}^\top \mathcal{Q}^n(\xi) \\
&= \gamma_{\text{\tiny W}}^{\star T} \Gamma_{\text{\tiny W}}^{-1} \mathcal{Q}^n(\xi) +(1- \mathbb{1}_n^\top \Gamma_{\text{\tiny W}}^{-1}\gamma_{\text{\tiny W}}^{\star}) \frac{\mathbb{1}^\top\Gamma_{\text{\tiny W}}^{-1} \mathcal{Q}^n(\xi) }{\mathbb{1}_n^\top \Gamma_{\text{\tiny W}}^{-1}\mathbb{1}_n} .\\
\end{align*}

We can rewrite it with $\beta(\xi) = \frac{\mathbb{1}^\top\Gamma_{\text{\tiny W}}^{-1} \mathcal{Q}^n(\xi) }{\mathbb{1}_n^\top \Gamma_{\text{\tiny W}}^{-1}\mathbb{1}_n}$. 

\begin{align*}
    \hat{Q}(\xi) &= \gamma_{\text{\tiny W}}^{\star T} \Gamma_{\text{\tiny W}}^{-1} \mathcal{Q}^n(\xi) +(1- \mathbb{1}_n^\top \Gamma_{\text{\tiny W}}^{-1}\gamma_{\text{\tiny W}}^{\star})\beta(\xi) \\
    &= \gamma_{\text{\tiny W}}^{\star T} \Gamma_{\text{\tiny W}}^{-1} \mathcal{Q}^n(\xi) +(1- \gamma_{\text{\tiny W}}^{\star T} \Gamma_{\text{\tiny W}}^{-1}\mathbb{1}_n)\beta(\xi)\\
    &= \gamma_{\text{\tiny W}}^{\star T} \Gamma_{\text{\tiny W}}^{-1} \mathcal{Q}^n(\xi) + \beta(\xi) -\gamma_{\text{\tiny W}}^{\star T} \Gamma_{\text{\tiny W}}^{-1}\mathbb{1}_n\beta(\xi) \\
    &= \beta(\xi) + \gamma_{\text{\tiny W}}^{\star T} \Gamma_{\text{\tiny W}}^{-1}(\mathcal{Q}^n(\xi) -\mathbb{1}_n \beta(\xi)) .
\end{align*}
We recall that \((\gamma_{\text{\tiny W}}^{\star })_i = \gamma^{\text{\tiny W}}(\x - x_i)\), where \(\x\) is the point at which the Kriging estimate is evaluated, as defined in (\ref{wass_krig}), and \(x_i\) represents the \(i^{\text{th}}\) observation point. Consequently, the Kriging estimator can be interpreted as a function of the estimation location \(\x\).

\begin{align}
    \hat{Q}(\xi) = \beta(\xi)+ \sum_{i=1}^n b_i(\xi) \gamma^{\text{\tiny W}}(\x-x_i)
    \label{Krig_form2}
\end{align}
where $b_i(\xi) = \left[\Gamma_{\text{\tiny W}}^{-1}(\mathcal{Q}^n(\xi) -\mathbb{1}_n \beta(\xi))\right]_i $.

Since Kriging is an interpolation method, the estimator is required to reproduce the observed values exactly at the observation points. This requirement leads to the following equalities:
\begin{align*}
    Q_j(\xi) = \sum_{i=1}^n b_i(\xi)\gamma^{\text{\tiny W}}(x_j - x_i) +\beta(\xi), (\forall j\in(1,...,n)).
\end{align*}

Let $b(\xi)$ be the vector of size n composed of the $b_i(\xi)$. Notice that:
\begin{align*}
    \mathbb{1}_n b(\xi) &= \mathbb{1}_n \Gamma_{\text{\tiny W}}^{-1}(\mathcal{Q}^n(\xi) -\mathbb{1}_n \beta(\xi))\\
    &= \mathbb{1}_n \Gamma_{\text{\tiny W}}^{-1}\mathcal{Q}^n(\xi) - \mathbb{1}_n \Gamma_{\text{\tiny W}}^{-1}\mathbb{1}_n\frac{\mathbb{1}^\top\Gamma_{\text{\tiny W}}^{-1} \mathcal{Q}^n(\xi) }{\mathbb{1}_n^\top \Gamma_{\text{\tiny W}}^{-1}\mathbb{1}_n}\\
    &=0.
\end{align*}

With the slight abuse of notation $\gamma^{\text{\tiny W}}(x_i,x_j) = \gamma^{\text{\tiny W}}_ {i,j}$ this raises the following matrix system for all $\xi \in (0,1)$:
$$
\begin{bmatrix} \gamma^{\text{\tiny W}}_ {1,1} &\cdots &\gamma^{\text{\tiny W}}_{1,n}& 1\\
\vdots& & \vdots & \vdots\\
\gamma^{\text{\tiny W}}_{n,1} &\cdots  &\gamma^{\text{\tiny W}}_{n,n}& 1\\
1 & \cdots & 1 & 0\\
\end{bmatrix}
\begin{bmatrix} b_1(\xi)\\
\vdots\\
b_n(\xi)\\
\beta(\xi)\\
\end{bmatrix}
= \begin{bmatrix} Q_1(\xi)\\
\vdots\\
Q_n(\xi)\\
0\\
\end{bmatrix}.
$$

If we suppose $Q_i(\xi)$ unknown for $(i\in\{1,...,n\})$, for $\beta^{[i]}(\xi) = \frac{\mathbb{1}_{n-1}^\top(\Gamma_{\text{\tiny W}}^{-1})^{[i]}\mathcal{Q}^{[i]}(\xi) }{\mathbb{1}_{n-1}^\top (\Gamma_{\text{\tiny W}}^{-1})^{[i]}\mathbb{1}_{n-1}}$, $b^{[i]}(\xi) = (\Gamma_{\text{\tiny W}}^{-1})^{[i]}(\mathcal{Q}^{[i]}(\xi) -\mathbb{1}_{n-1} \beta^{[i]}(\xi))$ where $(\Gamma_{\text{\tiny W}}^{-1})^{[i]}$ denotes the matrix obtained by deleting the $ith$ row and the $ith$ column of $\Gamma_{\text{\tiny W}}^{-1}$ and letting $ \mathcal{Q}^{[i]}$ be the $ith$ subvector of $\mathcal{Q}_{n}$:

\begin{align*}
\begin{bmatrix} \gamma^{\text{\tiny W}}_{1,1} & \cdots & \gamma^{\text{\tiny W}}_{1,i-1} &\gamma^{\text{\tiny W}}_{1,i+1} & \cdots &\gamma^{\text{\tiny W}}_{1,n}& 1\\
\vdots& & \vdots & \vdots & &\vdots & \vdots\\
\gamma^{\text{\tiny W}}_{i-1,1} & \cdots & \gamma^{\text{\tiny W}}_{i-1,i-1} &\gamma^{\text{\tiny W}}_{i-1,i+1} & \cdots &\gamma^{\text{\tiny W}}_{i-1,n}& 1\\
\gamma^{\text{\tiny W}}_{i+1,1} & \cdots & \gamma^{\text{\tiny W}}_{i+1,i-1} &\gamma^{\text{\tiny W}}_{i+1,i+1} & \cdots &\gamma^{\text{\tiny W}}_{i+1,n}& 1\\
\vdots& & \vdots & \vdots & &\vdots & \vdots\\
\gamma^{\text{\tiny W}}_{n,n} & \cdots & \gamma^{\text{\tiny W}}_{n,i-1} &\gamma^{\text{\tiny W}}_{n,i+1} & \cdots &\gamma^{\text{\tiny W}}_{n,n}& 1\\
1 & \cdots & 1 &1 & \cdots &1& 0\\
\end{bmatrix}
\begin{bmatrix} b^{[i]}_1\\
\vdots\\
b^{[i]}_{i-1}(\xi)\\
b^{[i]}_{i+1}(\xi)\\
\vdots\\
b^{[i]}_n(\xi)\\
\beta^{[i]}(\xi)\\
\end{bmatrix}
= \begin{bmatrix} Q_1(\xi)\\
\vdots\\
Q_{i-1}(\xi)\\
Q_{i+1}(\xi)\\
\vdots\\
Q_n(\xi)\\
0\\
\end{bmatrix}
.
\end{align*}

Using the formula (\ref{Krig_form2}) for the point $x_i$ with the $n-1$ other observations, we have:
$$\hat{Q}_i(\xi) = \beta^{[i]}(\xi) + \sum_{\substack{j=1  \\ j \ne i}}^{n} b^{[i]}_j(\xi)\gamma^{\text{\tiny W}}(x_j - x_i).
$$

We can now add this to the previous matrix system. This yields:
$$
\underbrace{
\begin{bmatrix} 
\gamma^{\text{\tiny W}}_{1,1}   &\cdots    &\gamma^{\text{\tiny W}}_{1,i-1}  &\gamma^{\text{\tiny W}}_{1,i}   &\gamma^{\text{\tiny W}}_{1,i+1}  &\cdots&\gamma^{\text{\tiny W}}_{1,n}  &1\\
\vdots         &          &\vdots          &\vdots         &\vdots          
&      &\vdots        &\vdots\\
\gamma^{\text{\tiny W}}_{i-1,1} &\cdots    &\gamma^{\text{\tiny W}}_{i-1,i-1}&\gamma^{\text{\tiny W}}_{i-1,i} &\gamma^{\text{\tiny W}}_{i-1,i+1}&\cdots&\gamma^{\text{\tiny W}}_{i-1,n}&1\\

\gamma^{\text{\tiny W}}_{i,1}   &\cdots    &\gamma^{\text{\tiny W}}_{i,i-1}  &\gamma^{\text{\tiny W}}_{i,i}   &\gamma^{\text{\tiny W}}_{i,i+1}  &\cdots&\gamma^{\text{\tiny W}}_{i,n}&1\\

\gamma^{\text{\tiny W}}_{i+1,1} &\cdots    &\gamma^{\text{\tiny W}}_{i+1,i-1}&\gamma^{\text{\tiny W}}_{i+1,i}&\gamma^{\text{\tiny W}}_{i+1,i+1}&\cdots&\gamma^{\text{\tiny W}}_{i+1,n}& 1\\
\vdots         &          & \vdots         & \vdots       &\vdots          &      & \vdots& \vdots\\
\gamma^{\text{\tiny W}}_{n,n} & \cdots & \gamma^{\text{\tiny W}}_{n,i-1}&\gamma^{\text{\tiny W}}_{n,i} &\gamma^{\text{\tiny W}}_{n,i+1} & \cdots &\gamma^{\text{\tiny W}}_{n,n}& 1\\
1 & \cdots & 1 &1 &1& \cdots &1& 0\\
\end{bmatrix}
}_{\begin{bmatrix} \Gamma_{\text{\tiny W}} & \mathbb{1}_n \\ \mathbb{1}_n^\top & 0 \end{bmatrix}}
\begin{bmatrix} b_1^{[i]}(\xi)\\
\vdots\\
b^{[i]}_{i-1}(\xi)\\
0\\
b^{[i]}_{i+1}(\xi)\\
\vdots\\
b^{[i]}_n\\
\beta^{[i]}(\xi)\\
\end{bmatrix}
= \begin{bmatrix} Q_1(\xi)\\
\vdots\\
Q_{i-1}(\xi)\\
\hat{Q}_i(\xi)\\
Q_{i+1}(\xi)\\
\vdots\\
Q_n(\xi)\\
0\\
\end{bmatrix}.
$$

Since $\Gamma_W$ is a positive semi-definite matrix, it is invertible, so the block matrix $\begin{bmatrix} \Gamma_{\text{\tiny W}} & \mathbb{1}_n \\ \mathbb{1}_n^\top & 0 \end{bmatrix}$ is invertible \cite{FUZHENZHANG} and its inverse is a block matrix of the form:
$$
\begin{bmatrix}  \tilde{\Gamma}& B \\ B^\top & 0 \end{bmatrix}
$$
where $\tilde{\Gamma} = \Gamma_{\text{\tiny W}}^{-1} - \Gamma_{\text{\tiny W}}^{-1}\mathbb{1}_n(\mathbb{1}_n^\top\Gamma_{\text{\tiny W}}^{-1}\mathbb{1}_n)^{-1}\mathbb{1}_n^\top\Gamma_{\text{\tiny W}}^{-1}$.

By taking the $i^{th}$ line of the system,we have: 
\begin{align*}
    0 &= \sum_{\substack{j=1  \\ j \ne i}}^{n} \tilde{\Gamma}_{i,j}Q_j(\xi) + \tilde{\Gamma}_{i,i}\hat{Q}_i(\xi).
\end{align*}

Finally we have:
\begin{align*}
    \hat{Q}_i(\xi) = - \sum_{\substack{j=1  \\ j \ne i}}^{n} \frac{\tilde{\Gamma}_{i,j}}{\tilde{\Gamma}_{i,i}} Q_j(\xi)
\end{align*}

and 
\begin{align*}
    \hat{Q}_i(\xi) - Q_i(\xi)= - \sum_{\substack{j=1}}^{n} \frac{\tilde{\Gamma}_{i,j}}{\tilde{\Gamma}_{i,i}} Q_j(\xi).
\end{align*}

Since every step is true for all $\xi \in (0,1)$ and $\tilde{\Gamma}$ does not depend on $\xi$, we have:
\begin{align*}
    \hat{Q}_i - Q_i= - \sum_{\substack{j=1}}^{n} \frac{\tilde{\Gamma}_{i,j}}{\tilde{\Gamma}_{i,i}} Q_j.
\end{align*}
\end{proof}

\begin{proof}[Proof of Corollary \ref{corollary}]

The proof of the corollary arises directly from the virtual formula \eqref{line:virtual formula} since:
\begin{align*}
    \hat\sigma_i^2 &=\EE \left[\int_0^1\left(Q_{\mu(x_i)}(\xi)-\hat{Q}^{(-i)}_{\mu(x_i)}(\xi)\right)^2d\xi\right]\\
    &=\EE \left[\int_0^1\left(- \sum_{\substack{j=1}}^{n} \frac{\tilde{\Gamma}_{i,j}}{\tilde{\Gamma}_{i,i}} Q_j(\xi)\right)^2d\xi\right]\\
    &=\EE \left[\int_0^1 \sum_{j=1}^n \sum_{k=1}^n \frac{\tilde{\Gamma}_{i,j}\tilde{\Gamma}_{i,k}}{\tilde{\Gamma}_{i,i}^2}Q_j(\xi)Q_k(\xi)d\xi\right].
\end{align*}
Thanks to the calculation of $T_2$ from Appendix \ref{app: proof of the Kriging system solution}, we have: 
\begin{align*}
    \hat\sigma_i^2&= -\sum_{j=1}^n \sum_{k=1}^n \frac{\tilde{\Gamma}_{i,j}\tilde{\Gamma}_{i,k}}{\tilde{\Gamma}_{i,i}^2}\gamma^{\text{\tiny W}}(x_i,x_j) + \sum_{j=1}^n\tilde{\Gamma}_{i,j}\EE\int_0^1 \left(Q_j(\xi)\right)^2d\xi.
\end{align*}
By stationarity, $\EE\int_0^1 \left(Q_j(\xi)\right)^2d\xi$ is constant; we denote it $C$. \\
Moreover, notice that $\tilde{\Gamma}\times\mathbb{1}_n =\Gamma_{\text{\tiny W}}^{-1}\mathbb{1}_n - \Gamma_{\text{\tiny W}}^{-1}\mathbb{1}_n(\mathbb{1}_n^\top\Gamma_{\text{\tiny W}}^{-1}\mathbb{1}_n)^{-1}\mathbb{1}_n^\top\Gamma_{\text{\tiny W}}^{-1}\mathbb{1}_n = 0$.
That shows that $\sum_{j=1}^n\tilde{\Gamma}_{i,j}\EE\int_0^1 \left(Q_j(\xi)\right)^2d\xi = 0.$

So
\begin{align*}
    \hat\sigma_i^2&= -\sum_{j=1}^n \sum_{k=1}^n \frac{\tilde{\Gamma}_{i,j}\tilde{\Gamma}_{i,k}}{\tilde{\Gamma}_{i,i}^2}\gamma^{\text{\tiny W}}(x_j,x_k)\\
    &=-\sum_{j=1}^n\frac{\tilde{\Gamma}_{i,j}}{\tilde{\Gamma}_{i,i}^2}\sum_{k=1}^n\tilde{\Gamma}_{i,k}\gamma^{\text{\tiny W}}(x_j,x_k)\\
    &=-\sum_{j=1}^n\frac{\tilde{\Gamma}_{i,j}}{\tilde{\Gamma}_{i,i}^2}\sum_{k=1}^n\tilde{\Gamma}_{i,k}\Gamma_{\text{\tiny W},k,j}.
\end{align*}

In the previous proof, we defined the matrix $\begin{bmatrix} \Gamma_{\text{\tiny W}} & \mathbb{1}_n \\ \mathbb{1}_n^\top & 0 \end{bmatrix}$ and its inverse $\begin{bmatrix}  \tilde{\Gamma}& B \\ B^\top & 0 \end{bmatrix}$.
By denoting $I_n$ the identity matrix of size $n$ we have: 
\begin{align*}
    I_{n+1} = \begin{bmatrix} \Gamma_{\text{\tiny W}} & \mathbb{1}_n \\ \mathbb{1}_n^\top & 0 \end{bmatrix}\begin{bmatrix}  \tilde{\Gamma}& B \\ B^\top & 0 \end{bmatrix}.
\end{align*}
So
$$
I_n = \Gamma_{\text{\tiny W}}\tilde{\Gamma} + \mathbb{1}_n^\top B^\top.$$
It shows that 
$$\sum_{k=1}^n\tilde{\Gamma}_{i,k}\Gamma_{\text{\tiny W},k,j} = \delta_{i,j} - B_i,$$
where $\delta_{i,j} =1 $ if $i=j$, $0$ if $i\neq j$.

Then 
\begin{align*}
    \hat\sigma_i^2&=-\sum_{j=1}^n\frac{\tilde{\Gamma}_{i,j}}{\tilde{\Gamma}_{i,i}^2}(\delta_{i,j} - B_i)\\
    &= - \frac{1}{\tilde{\Gamma}_{i,i}}+\frac{B_i}{\tilde{\Gamma}_{i,i}^2}\sum_{j=1}^n\tilde{\Gamma}_{i,j}.
\end{align*}
As shown before, $\tilde{\Gamma}\times\mathbb{1_n} = 0$, so
\begin{align*}
    \hat\sigma_i^2 = - \frac{1}{\tilde{\Gamma}_{i,i}}.
\end{align*}
\end{proof}

\bibliographystyle{plain}

\end{document}